\documentclass[10pt]{article}
\usepackage{amsfonts,amssymb,amstext}
\usepackage{amsmath}
\usepackage{amstext}
\usepackage{amsthm}
\makeatletter

\newtheorem{thm}{Theorem}
\newtheorem{lem}{Lemma}

\newtheorem{exa}{Example}
\newtheorem{rem}{Remark}

\newtheorem*{TA}{Theorem A}
\newtheorem*{TB}{Theorem B}

\begin{document}

\title{On the trace approximations of products
of Toeplitz matrices\footnote{Published in Statistics and Probability Letters: 83(2013)753-760.}}

\author{M. S. Ginovyan\footnote{M. Ginovyan research
was partially supported by National Science Foundation
Grant \#DMS-0706786}\footnote{Department of Mathematics and Statistics,
Boston University, e-mail: ginovyan@math.bu.edu}
\, and A. A. Sahakyan\footnote{Department of Mathematics and Mechanics,
Yerevan State University, e-mail: sart@ysu.am}}
\date{}

\maketitle

\begin{abstract}
The paper establishes error orders for integral
limit approximations to the traces of products of
Toeplitz matrices generated by integrable real symmetric
functions defined on the unit circle.
These approximations and the corresponding error bounds are
of importance in the statistical analysis of discrete-time
stationary processes: asymptotic distributions and large
deviations of Toeplitz type random quadratic forms, estimation of
the spectral parameters and functionals, etc.
\end{abstract}

Spectral density.


\section{Introduction}

Toeplitz matrices, which have great independent interest and a
wide range of applications in different fields of science
(economics, engineering, finance, hydrology, physics, etc.),
arise naturally in the statistical analysis of
stationary processes -
the covariance matrix of a discrete-time stationary process
is a Toeplitz matrix generated by the spectral density
of that process, and vice versa, any non-negative summable function
generates a Toeplitz matrix, which can be considered as a spectral
density of some discrete-time stationary process, and therefore
the corresponding Toeplitz matrix will be the covariance
matrix of that process.

The present paper is devoted to the problem of approximation of
the traces of products of Toeplitz matrices generated
by integrable real symmetric functions defined on the unit circle,
and estimation of the corresponding errors.

The trace approximation problem and its applications in the
statistical analysis and prediction of discrete-time stationary
processes go back to the classical monograph by \cite{GS}.
Later this problem for different classes of generating functions (symbols)
has been considered by many authors (see, e.g., \cite{R2},  \cite{I},
\cite{Ta}, \cite{FT1}, \cite{A}, \cite{G0}, \cite{Dah},
\cite{GSu}, \cite{TK}, \cite{LP}, \cite{GS1}, \cite{GKSu}, and references therein).
Notice that the trace approximation problem is of particular importance
in the cases where the symbols of the underlying Toeplitz matrices
have singularities. For instance,
such cases arise in many problems of statistical analysis
(asymptotic distributions and large deviations of Toeplitz type random
quadratic forms, estimation of the spectral parameters and functionals, etc.)
of long-memory (the spectral density is unbounded) and anti-persistent
(the spectral density has zeros) discrete-time stationary processes
(see, e.g., \cite{IH1}, \cite{FT1}, \cite{A}, \cite{Dah}, \cite{GSu}, \cite{BGR},
\cite{TK}, \cite{LP}, \cite{GS1}), \cite{GKSu}).

The paper is organized as follows.
In the remainder of this section we review and summarize some known
results concerning trace approximation problem.
In Section 2 we state the main results of the paper
and discuss two examples.
Section 3 is devoted to the proofs of results stated in Section 2.

Throughout the paper the letters $C$, $c$ and $M$, with or without index,
are used to denote positive constants, the values of which can vary from line to line. Also, all functions considered in this paper are assumed to be $2\pi$-periodic and periodically extended to $\mathbb{R}$.

Let $f(\lambda)$ and $g(\lambda)$ be integrable real symmetric functions
defined on
$\mathbb{T}:\, =[-\pi, \pi]$, and let $T_n(f)$ and $T_n(g)$ be the $(n\times n)$
Toeplitz matrices generated by functions $f(\lambda)$ and $g(\lambda)$,
respectively: for $u(\lambda)\in L^1(\mathbb{T})$ we define
\begin{equation}
\label{in-1}
T_n(u)=\|\widehat u(k-j)\|_{k,j=\overline{1,n}},
\quad n=1, 2,\ldots,
\end{equation}
where
\begin{equation}
\label{in-02}
\widehat u(k)=\int_\mathbb{T} e^{i\lambda k}\,u(\lambda)\,d\lambda,
\quad k\in \mathbb{Z}
\end{equation}
are the Fourier coefficients of $u(\lambda)$.

Let $\nu$ be an arbitrary fixed positive integer. Define
\begin{equation}
\label{in-2} S_{n,\nu}: \, =S_{n,\nu}(f,g)=\frac1n\hbox{tr}[T_n(f)T_n(g)]^\nu,
\end{equation}
\begin{equation}
\label{in-3}
M_\nu: \, = M_\nu(f,g)=(2\pi)^{2\nu-1}\int_{-\pi}^{\pi}[f(\lambda)g(\lambda)]^\nu\,d\lambda
\end{equation}
and set
\begin{equation}
\label{in-4} \Delta_{n,\nu}: \, =\Delta_{n,\nu}(f,g)=|S_{n,\nu}-M_\nu|.
\end{equation}

The problem is to approximate $S_{n,\nu}$ by $M_{\nu}$ and estimate
the error rate for $\Delta_{n,\nu}$.
More precisely, find conditions on functions $f(\lambda)$ and $g(\lambda)$
such that:
\begin{eqnarray}
\label{im-4}
&&{\rm (A):} \quad \Delta_{n,\nu}(f,g)=o(1) \quad {\rm as} \quad n\to\infty, \quad{\rm or}\\
\label{im-5}
&&{\rm (B):} \, \quad \Delta_{n,\nu}(f,g)=O(n^{-\gamma}) \quad {\rm for \,\,some } \quad \gamma>0
\quad {\rm as} \quad n\to\infty.
\end{eqnarray}

In Theorems A and B below we summarize some known results
concerning Problems (A) and (B), respectively.

\begin{TA}
\label{TA}
Each of the following conditions is sufficient for
$$
\Delta_{n,2}(f,g)=|S_{n,2}-M_2|=o(1) \quad {\rm as} \quad n\to\infty.
$$

\begin{itemize}
\item[{\bf (A1)}]
$f(\lambda)\in L^p(\mathbb T)$ $(p\ge1)$ and  $g(\lambda)\in L^q(\mathbb T)$ $(q\ge1)$
with $1/p+1/q\le 1/2$.

\item[{\bf (A2)}]
$f\in  L^2(\mathbb T)$, \,$g\in  L^2(\mathbb T)$, $fg\in L^2(\mathbb T)$ and
\begin{equation}
\label{in-7} \int_{\mathbb T}
 f^2(\lambda)g^2(\lambda-\mu)\,d\lambda \longrightarrow
\int_{\mathbb T} f^2(\lambda)g^2(\lambda)\,d\lambda
\quad {\rm as} \quad \mu\to0.
\end{equation}

\item[{\bf (A3)}]
The function
\begin{equation}\label{in-6}
\varphi({\bf u})=\varphi(u_1, u_2,u_3)=\int_\mathbb T
f(\lambda)g(\lambda-u_1)f(\lambda-u_2)g(\lambda-u_3)\,d\lambda,\quad
\end{equation}
belongs to $L^2(\mathbb T^3)$ and is continuous at ${\bf 0}=(0,0,0)$.

\item[{\bf (A4)}]
$f(\lambda)\le |\lambda|^{-\alpha}L_1(\lambda)$ and
$|g(\lambda)|\le |\lambda|^{-\beta}L_2(\lambda)$ for
$\lambda\in {\mathbb T}$ and some
$\alpha<1,\ \beta<1,$ $\alpha+\beta\le1/2,$ and
$
L_i\in SV(\mathbb{R}), $
$\lambda^{-(\alpha+\beta)}L_i(\lambda)\in L^2(\mathbb{T}),$ \ $i=1,2,$
where $SV(\mathbb{R})$ is the class of slowly varying
at zero functions $u(\lambda)$, $\lambda\in\mathbb{R}$ satisfying
$u(\lambda)\in L^\infty(\mathbb{R}),$\
$\lim_{\lambda\to0}u(\lambda)=0,$ \
$u(\lambda)=u(-\lambda)$ and $0<u(\lambda)<u(\mu)$\ for\ $0<\lambda<\mu.$
\end{itemize}
\end{TA}

\begin{rem}
{\rm
Assertion (A1) was proved by \cite{A}.
For special case $p=q=\infty$,
it was first established by \cite{GS},
while the case $p=2$, $q=\infty$ was proved by
\cite{R2} and \cite{I}.
Assertion (A2) was proved in \cite{GSu}
(see, also, \cite{GKSu}).
Assertions (A3) and (A4) were established in \cite{GS1}.
A special case of (A4), when
$\alpha+\beta<1/2$, was considered by \cite{FT1}.}
\end{rem}

\begin{TB}
\label{TB}
The following assertions hold:
\begin{itemize}
\item[{\bf (B1)}] 
If the Fourier coefficients  $\widehat f(k)$ and $\widehat g(k)$
of functions $f(\lambda)$ and  $g(\lambda)$ satisfy the conditions
\begin{equation}
\label{im-8}
\sum_{k=-\infty}^\infty|k||\widehat f(k)|<\infty \quad {\rm and} \quad
\sum_{k=-\infty}^\infty|k||\widehat g(k)|<\infty,
\end{equation}
then for $\nu=1,2,\ldots$
\begin{equation}
\label{im-9}
\Delta_{n,\nu}(f,g)=O(n^{-1}) \quad  {\rm as }\ n\to\infty.
\end{equation}

\item[{\bf (B2)}]
If there exist constants $C_i$ with $0<C_i<\infty,$ $i=1,2,3,4,$ such that
\begin{eqnarray}
\label{LP-0}
&&
\sup_{\lambda\in [-\pi,\pi]}|f(\lambda)|\le C_1, \quad
\sup_{\lambda\in [-\pi,\pi]}|g(\lambda)|\le C_2,\\
&&
\label{LP-1}
\sup_{\lambda\in [-\pi,\pi]}|f'(\lambda)|\le C_3, \quad
\sup_{\lambda\in [-\pi,\pi]}|g'(\lambda)|\le C_4,
\end{eqnarray}
then for any \ $\epsilon>0$ and $\nu=1,2,\ldots$
\begin{equation}
\label{in-101}
\Delta_{n,\nu}(f,g)=O(n^{-1+\epsilon}) \quad  as \,\, n\to\infty.
\end{equation}
\end{itemize}
\end{TB}

\begin{rem}
{\rm
Assertion (B1) was established in \cite{Ta} (see, also, \cite{TK}).
Assertions (B2) was proved in \cite{LP}.
Note that in (B2) the asymptotic relation (\ref{in-101})
is valid under the single condition (\ref{LP-1})
because (\ref{LP-1}) obviously implies (\ref{LP-0}).}
\end{rem}

\begin{rem}
{\rm
In  \cite {LP} was also stated the following result (see \cite {LP}, Theorem 2)}.

\noindent
{\bf (B3)}
Assume that the functions $f(\lambda)$ and $g(\lambda)$ satisfy the conditions:
\begin{itemize}
\item[(a)]
$f(\lambda)$ and $g(\lambda)$ are symmetric, real valued, continuously differentiable
at all $\lambda\ne 0$ and there exist $0<C_{i}<\infty$, $i=1,2,$ such that for any
$\lambda\in [-\pi,\pi]$
$$|f(\lambda)|\le C_1|\lambda|^{-\alpha}, \quad |g(\lambda)|\le C_2|\lambda|^{-\beta},
\quad \alpha<1, \, \beta<1.$$
\item[(b)]
For all $t>0$ there exist $M_{t1}$ and $M_{t2}$ such that
$$\sup_{|\lambda|>t}|f'(\lambda)|\le M_{t1} \quad {\rm and} \quad \sup_{|\lambda|>t}|g'(\lambda)|\le M_{t2}.$$
\item[(c)]
$\nu(\alpha+\beta)<1$, \ $\nu\in\mathbb{N}$.
\end{itemize}
Then for any \ $\epsilon>0$
\begin{equation}
\label{im-10}
\Delta_{n,\nu}(f,g)= \left \{
           \begin{array}{ll}
           O(n^{-1+\nu(\alpha+\beta)+\epsilon}), & \mbox {if \, $\alpha+\beta>0$}\\
           O(n^{-1+\epsilon}), & \mbox {if \, $\alpha+\beta\le0$}.
           \end{array}
           \right.
\end{equation}

{\rm First observe that condition (a) in (B3) implies condition (b).

Unfortunately, the proof of (B3) given in \cite {LP} contains an inaccuracy.
The issue is that the authors assertion that
{\em "the last integral in formula (26) is finite under the conditions (27)"}
(\cite{LP}, p. 743), is not correct.

More precisely, they state that for some $t\in (0,\pi)$  the integral
\begin{eqnarray}
\label{LP-3}
\nonumber
&&I:\,=\int_{A_t}
|z_1|^{2\nu \eta-\nu(\alpha+\beta)-1}
|z_2\cdots z_{2\nu}-1|^{\eta-1}\\
&&\hskip25mm
\times |z_2-1|^{\eta-1}\cdots|z_{2\nu}-1|^{\eta-1}
dz_1\cdots dz_{2\nu},
\end{eqnarray}
where
\begin{eqnarray}
\label{LP-33}
\nonumber
A_t:=\big\{(z_1,\ldots, z_{2\nu})\in \mathbb R^{2\nu}:\ |z_1|\leq t,\  |z_1z_2|\leq t,
\ldots, \ |z_1\cdots z_{2\nu}|\leq t, &&\\
|z_1|>\frac12|z_1z_2|>\cdots>\frac1{2^{2\nu-1}}|z_1\cdots z_{2\nu}|>\frac1{2^{2\nu}}|z_1|\big\}&&
\end{eqnarray}
converges (is finite) in the parameter set (see (c) and  \cite {LP},  formulas (26) and (27)):
\begin{equation}
\label{LP-4}
\frac12( \alpha+\beta)
<\eta<1, \qquad 0< \alpha+\beta<\frac1\nu,
\end{equation}
and then conclude that the quantity
\begin{equation}
\label{LP-5}
J_n:=\frac C{n^{1-2\nu\eta}}\times I,
\end{equation}
where $C$ is a constant, goes to zero as $n\to\infty$ with the specified rate.

First observe that to have $J_n\to 0$ as $n\to\infty$, the condition
$1-2\nu\eta > 0$ should be imposed, that is, along with (\ref{LP-4}),
the parameter $\eta$ should also satisfy
\begin{equation}
\label{LP-6}
\eta<\frac 1 {2\nu}.
\end{equation}

The arguments that follow, show that the integral in (\ref{LP-3})
diverges in the parameter set (\ref{LP-4}), (\ref{LP-6}).

We first prove the following inequality: for $ 0<\gamma<1,\  \ 0<\theta<1$ and $y_0>2$
\begin{equation}
\label{LP-61}
\int_1^2  { (xy-1)^{-\theta} (x-1)^{-\gamma}} dx\ge  c\,
{(y-1)^{1-\gamma-\theta}},\qquad 1<y<y_0,
\end{equation}
where the constant  $c$ depends only on $\gamma$,  $\theta$ and $y_0$.
To prove (\ref{LP-61}), observe first that for $1\le x\le y\le 2$,
$$
xy-1=(x-1)(y-1)+(x-1)+(y-1)\le 3(y-1).
$$
Consequently,
\begin{eqnarray*}
&&\int_1^2   (xy-1)^{-\theta}} {(x-1)^{-\gamma}dx\ge
\int_1^y   (xy-1)^{-\theta}} {(x-1)^{-\gamma}dx\ge\\
&\ge&
{3^{-\theta}}{(y-1)^{-\theta}}\int_1^y  {(x-1)^{-\gamma}} dx
= \frac{3^{-\theta}} {(1-\gamma)}\cdot
{(y-1)^{1-\gamma-\theta}},
\end{eqnarray*}
yielding  (\ref{LP-61}) for $1< y\le2$.

For $2<y<y_0$ we have ${(y-1)^{1-\gamma-\theta}}<y_0$  and
\begin{eqnarray*}
\int_1^2  (xy-1)^{-\theta}}  {(x-1)^{-\gamma}dx&\ge&
\int_1^2  (xy_0-1)^{-\theta}}  {(x-1)^{-\gamma}dx\\
&=:&J\ge \frac {J}{y_0}\cdot
{(y-1)^{1-\gamma-\theta}},
\end {eqnarray*}
where $J$ depends only on $\gamma$,  $\theta$ and $y_0$. Inequality (\ref{LP-61})  is proved.

Now, setting $\varepsilon:=2\nu\eta-\nu(\alpha+\beta)>0$,
taking into account that  (see (\ref{LP-33}))
$$
A_t\supset\left\{(z_1,\ldots z_{2\nu})\in \mathbb R^{2\nu}:0<z_1< \frac t{2^\nu}, \ 1<z_i<2, \ i=2,3,\ldots,2\nu\right\},
$$
and applying (\ref {LP-61}) with $y_0=2^{2\nu}$ successively $(2\nu-2)$ times,
we get
\begin{eqnarray*}
I&=&
\int_{A_t}|z_1|^{\varepsilon-1}
|z_2\cdots z_{2\nu}-1|^{\eta-1}
|z_2-1|^{\eta-1}\cdots |z_{2\nu}-1|^{\eta-1}
dz_1\cdots dz_{2\nu}\\
&\ge &
\int_{0}^{t/2^\nu}z_1^{\varepsilon-1}dz_1
\int_1^2\cdots \int_1^2
(z_2z_3\cdots z_{2\nu}-1)^{\eta-1}(z_2-1)^{\eta-1}dz_2\\
&&\hskip4cm\times
(z_3-1)^{\eta-1} \cdots (z_{2\nu}-1)^{\eta-1}dz_3\cdots dz_{2\nu}\\
&\ge & c
\int_{0}^{t/2^\nu}z_1^{\varepsilon-1}dz_1
\int_1^2\cdots \int_1^2
(z_3\cdots z_{2\nu}-1)^{2\eta-1}(z_3-1)^{\eta-1}dz_3\\
&&\hskip4cm\times
(z_4-1)^{\eta-1} \cdots (z_{2\nu}-1)^{\eta-1}dz_4\cdots dz_{2\nu}\\
&\ge& \cdots \cdots\cdots \cdots \cdots\cdots \cdots \cdots\cdots \cdots \cdots\cdots\\
&\ge& c_1
\int_{0}^{t/2^\nu}z_1^{\varepsilon-1}dz_1
\int_1^2
(z_{2\nu}-1)^{2\nu\eta-2}dz_{2\nu}.
\end{eqnarray*}
The last integral diverges, since by (\ref {LP-6})  $2\nu\eta-2<-1$ .
}
\end{rem}

In this paper we prove the asymptotic relation
\begin{equation}
\label{in-10} \Delta_{n,\nu}(f,g) =O(n^{-\gamma}), \quad \gamma>0,
\quad {\rm as} \quad n\to\infty
\end{equation}
for some classes of generating functions $f(\lambda)$ and $g(\lambda)$.
The results improve some of the $o(1)$ rates stated in Theorem A.
For simplicity we state and prove the results in the typical special
case where $\nu=2.$

\section{Error bounds for $\Delta_{n,2}$}

For $\psi\in L^p(\mathbb T)$, $1\le p\le\infty$ we denote by
$\omega_p(\psi,\delta)$ the $L^p$--modulus of continuity of $\psi$:
$$
\omega_p(\psi,\delta) :=\sup_{0<h\le \delta}\|\psi(\cdot+h)-\psi(\cdot)\|_p,
\quad \delta>0.
$$
Given numbers $0 < \gamma \le 1$ and $1\le p\le\infty$, we denote by
${\rm Lip} (p, \gamma)={\rm Lip} (\mathbb T; p, \gamma)$ the $L^p$-Lipschitz
class of functions defined on $\mathbb T$ (see, e.g., \cite{BN}):
$${\rm Lip} (p, \gamma) = \{\psi(\lambda)\in L^p(\mathbb T);
\quad \omega_p(\psi;\delta) = O(\delta ^\gamma),
\quad \delta \to 0\}.$$
Observe that if $\psi\in{\rm Lip} (p, \gamma)$, then there exists a
constant $C$ such that $\omega_p(\psi;\delta)\le C\,\delta ^\gamma$
for all $\delta>0$.

The main results of the paper are the following theorems.

\begin{thm}\label{th01}
Let the function $\varphi ({\bf u})=\varphi(u_1,u_2,u_3)$ be as in
(\ref{in-6}). Assume that with some constants $C>0$ and
 $\gamma\in(0,1]$
\begin{equation} \label{t1}
|\varphi ({\bf
u})-\varphi ({\bf 0})|\leq C|{\bf u}|^\gamma, \quad {\bf u}=(u_1,u_2,u_3)\in
\mathbb R^3, \end{equation} where ${\bf 0}=(0,0,0)$ and
$|{\bf u}|=|u_1|+ |u_2| + |u_3|$. Then for any $\varepsilon >0$
\begin{equation} \label{t2}
\Delta_{n,2}(f,g)=O\left(n^{-\gamma+\varepsilon}\right)
\quad\text{as}\quad n\to\infty. \end{equation}
\end{thm}
The next two theorems we will deduce from Theorem \ref{th01},
and hence can be considered as corollaries of Theorem \ref{th01}.

\begin{thm}\label{th02}
Let $f\in{\rm Lip}(p, \gamma)$  and
$g\in {\rm Lip}(q, \gamma)$ with $\gamma \in (0,1]$ and
$p,q\ge1$ such that $1/p+1/q\le1/2$.
Then (\ref{t2}) holds for any $\varepsilon >0$.
\end{thm}

\begin{thm}\label{NT}
Let $f_i(\lambda)$, $i=1,2$, be two differentiable functions
on $[-\pi,\pi]\setminus\{0\}$, such that for some constants
$\alpha_i>0$, $i=1,2$, satisfying $\alpha_1+\alpha_2<1/2$
 and $M_{1i}, M_{2i}>0$,  $i=1,2$
\begin{equation}
\label{d1}
|f_i(\lambda)|\le M_{1i}|\lambda|^{-\alpha_i},\quad
|f_i^\prime(\lambda)|\le M_{2i}|\lambda|^{-(\alpha_i+1)},\quad
\lambda\in [-\pi,\pi]\setminus\{0\}.
\end{equation}
Then for any $\varepsilon >0$
\begin{equation} \label{t20}
\Delta_{n,2}(f_1,f_2)=O\left(n^{-\gamma+\varepsilon}\right)
\quad\text{as}\quad n\to\infty
\end{equation}
with
\begin{equation}
\label{d2}
\gamma =\frac 14-\frac{\alpha_1+\alpha_2}2.
\end{equation}
\end{thm}

\begin{exa}\label{exa1}
{\rm Let $f_i(\lambda)=|\lambda|^{-\alpha_i}$, $\lambda\in [-\pi,\pi],$
$i=1,2$, with $\alpha_1, \, \alpha_2>0$ and $\alpha_1+\alpha_2<1/2$.
It is easy to see that the conditions of Theorem \ref{NT} are satisfied,
and hence we have (\ref{t20}) with $\gamma $ as in (\ref{d2}).}
\end{exa}
\begin{exa}\label{exa2}
{\rm Let $f_i(\lambda)$, $\lambda\in [-\pi,\pi],$ $i=1,2$, be the spectral
density functions of two long-memory discrete-time stationary processes
given by
\begin{equation}\label{d20}
f_i(\lambda)=\frac{\sigma^2_i}{2\pi}|1-e^{i\lambda}|^{-\alpha_i}
\end{equation}
with $0<\sigma^2_i<\infty$, $\alpha_i>0$, $i=1,2$, and
$\alpha_1+\alpha_2<1/2$.
Then (\ref{t20}) holds with $\gamma $ as in (\ref{d2}).

Indeed, assuming  that $\lambda\in (0,\pi]$
(the case $\lambda\in [-\pi,0)$ is treated similarly), and taking into account $|1-e^{i\lambda}|=2\sin(\lambda/2)$, we have for $i=1,2$
$$
f_i(\lambda)= \frac{\sigma^2_i}{2\pi} \cdot 2^{-\alpha_i }\left[\sin\frac \lambda 2\right]^{-\alpha_i}
$$
and
$$
f_i^\prime(\lambda)=\frac{\sigma^2_i}{2\pi} \cdot \left[-\alpha_i2^{-\alpha_i-1}
\left(\sin\frac \lambda 2\right)^{-\alpha_i-1}\cos\frac \lambda 2\right].
$$
It is clear that the conditions of Theorem \ref{NT} are satisfied
with  $M_{1i}=M_{2i} =\sigma_i^2$,   $i=1,2$,
and the result follows.}
\end{exa}

\begin{rem}
{\rm It is easy to see that under the conditions of Theorem (B2)
we have $f\in{\rm Lip}(p, 1)$  and $g\in {\rm Lip}(p, 1)$ for any $p\ge 1$.
Hence Theorem \ref{th02} implies Theorem (B2) (for $\nu=2$).}
\end{rem}

\begin{rem}
{\rm For functions $f_i(\lambda)$, $i=1,2$, defined by
(\ref{d20}) an explicit second-order expansion for $S_{n,1}$
(see (\ref{in-2})) was found by \cite{LP},
where they showed that in this special case the second-order
expansion removes the singularity in the first-order
approximation, and provides an improved approximation
of order $\gamma = 1-2({\alpha_1+\alpha_2})$.}
\end{rem}

\section{Proofs}

We first state a number of lemmas.
The results of the first two lemmas are known (see, e.g., 
\cite{GKSu}, p. 8, 161).

\begin{lem}
\label{lem01} Let $D_n(u)$ be the Dirichlet kernel
\begin{equation}
\label{a-2} D_n(u)=\frac{\sin(nu/2)}{\sin (u/2)}.
\end{equation}
Then, for any $\delta\in[0,1]$ and $u\in \mathbb{T}$
\begin{equation} \label{01}
|D_n(u)|\leq  \pi\,{n^\delta}{|u|^{\delta-1}}.
\end{equation}
\end{lem}
\begin{lem}
\label{lem1}Let $0<\beta<1$, $0<\alpha<1$, and
$\alpha+\beta>1$. Then for any \ $y\in \mathbb R,\ y\neq0$,
\begin{equation}
\label{a-4}
\int_\mathbb R \frac1{|x|^\alpha|x+y|^\beta}dx = \frac M
{|y|^{\alpha+\beta-1}},
\end{equation}
where $M$ is a constant depending on $\alpha$ and $\beta$.
\end{lem}

\begin{lem}
\label{lem2} Let $0<\alpha\le 1$ and $\frac
2 3<\beta<\frac{\alpha+3}4$. Then
\begin{equation}
\label{a-05}
B_i:=\int_{\mathbb T^3}
\frac{|u_i|^\alpha}{|u_1u_2u_3(u_1+u_2+u_3)|^\beta} \,
du_1du_2du_3<\infty,
\quad i=1,2,3.
\end{equation}
\end{lem}

\begin{proof}
Using Lemma \ref{lem1} we can write
\begin{eqnarray*}
B_1&\leq&\int_{|u_1|\leq \pi}\frac1{|u_1|^{\beta-\alpha}}
\int_{\mathbb R}\frac1{|u_2|^\beta}\int_{\mathbb R}\frac1{|u_3(u_1+u_2+u_3)|^\beta} \,
du_3du_2du_1\\
&=& M\int_{|u_1|\leq\pi}\frac1{|u_1|^{\beta-\alpha}}
\int_{\mathbb R}\frac1{|u_2|^\beta|u_1+u_2|^{2\beta-1}}du_2du_1\\
&=&
M^2\int_{|u_1|\leq\pi}\frac1{|u_1|^{4\beta-\alpha-2}}du_1<\infty,
\end{eqnarray*}
yielding (\ref{a-05}) for $i=1$.
The quantities $B_2$ and $B_3$ can be estimated in the same way.
\end{proof}

\begin{lem}\label{lem4}
Let $p>1$ and $0<\alpha <1$ be such that $\alpha p<1$, and for
some constants $M_1, M_2>0$
\begin{equation}\label{l4-01}
|f(\lambda)|\le M_1|\lambda|^{-\alpha},\quad
|f^\prime(\lambda)|\le M_2|\lambda|^{-(\alpha+1)},\quad \lambda\in [-\pi,\pi],
\quad \lambda\ne0.
\end{equation}
Then $f\in Lip(p,1/p-\alpha)$.
\end{lem}
\begin{proof}[Proof]
Let $h\in (0,1)$ be fixed. Then
\begin{equation}\label{l4-1}
\int_{|\lambda|\le2h}|f(\lambda+h)-f(\lambda)|^pd\lambda  \le (2M_1)^p \int_0^{3h}
\lambda^{-p\alpha} d\lambda \le C h^{1-p\alpha}.
\end{equation}
Next, for $|\lambda|>2h$ we have with some $\xi\in(\lambda,\lambda+h)$
$$
|f(\lambda+h)-f(\lambda)|=|f^\prime(\xi)\cdot h|
\le M_2 h|\xi|^{-(\alpha+1)}.
$$
Hence
\begin{equation}\label{l4-2}
\int_{2h<|\lambda|<\pi}|f(\lambda+h)-f(\lambda)|^pd\lambda  \le C h^p\int_h^\pi
\lambda^{-p(\alpha+1)} d\lambda \le C h^{1-p\alpha}.
\end{equation}
From (\ref{l4-1}) and (\ref{l4-2}) we get
$$
\|f(\lambda+h)-f(\lambda)\|_p\le Ch^{1/p-\alpha},
$$
showing that $f\in Lip(p, 1/p-\alpha)$.
\end{proof}

\begin{proof}[Proof of Theorem \ref{th01}]
Denote
\begin{equation}
\label{a-1} \Phi_n({\bf u}):=\Phi_n(u_1, u_2, u_3)=\frac1{8\pi^3n}\cdot
D_n(u_1)D_n(u_2)D_n(u_3)D_n(u_1+u_2+u_3)
\end{equation}
 and
\begin{equation}
\label{w1} \Psi({\bf u}):=\Psi(u_1,u_2,u_3)
= \varphi(u_1,u_1+u_2,u_1+u_2+u_3),
\end{equation}
where $D_n(u)$ and $\varphi(u_1,u_2,u_3)$ are defined by (\ref{a-2})
and (\ref{in-6}), respectively.
Then (\cite{GS1})
\begin{eqnarray}
\label{a-3}
\nonumber
\Delta_{n,2}&=&\left|\frac1n \, \hbox{tr}[B_n(f)B_n(g)]^2
-8\pi^3\int_{\mathbb T}f^2(\lambda)g^2(\lambda)\,d\lambda\right| \\
&=&
\left |\int_{\mathbb T^3}[\Psi({\bf u})-\Psi({\bf 0})]
\Phi_n({\bf u})d{\bf u}\right|,\quad {\bf 0}=(0,0,0).
\end{eqnarray}
It follows from (\ref{t1}) and (\ref{w1}) that
\begin{equation}
\label{w2}
|\Psi({\bf u})-\Psi({\bf 0})|\leq 3C|u_1|^\gamma +2C|u_2|^\gamma+C|u_3|^\gamma,
\quad {\bf u}=(u_1, u_2, u_3)\in \mathbb{T}^3.
\end{equation}

\noindent
Let $\varepsilon\in (0,\gamma)$. Then, applying
Lemma \ref{lem01} with $\delta=\frac{1+\varepsilon-\gamma}4$,
and using (\ref{a-3}) and (\ref{w2}), we have
\begin{eqnarray*}
|\Delta_{n,2}|&\leq&
\int_{\mathbb T^3}|\Psi({\bf u})-\Psi({\bf 0})||\Phi_n({\bf u})|d{\bf u}
\\
&\leq& \frac2{n^{1-4\delta}} \sum\limits _{i=1}^3 C_i\int_{\mathbb T^3}
\frac{|u_i|^\gamma}{|u_1u_2u_3(u_1+u_2+u_3)|^{1-\delta}} \, du_1du_2du_3.
\end{eqnarray*}
This, combined with Lemma \ref{lem2} implies
the statement of Theorem \ref{th01}.
\end{proof}

\begin{proof}[Proof of Theorem \ref{th02}]
According to Theorem \ref{th01} it is enough to prove that the
function
\begin{equation} \label {c-2-0}
\varphi ({\bf t}):=\int_\mathbb T
h_0(u)h_1(u-t_1)h_2(u-t_2)h_3(u-t_3)du, \quad {\bf
t}=(t_1,t_2,t_3)\in\mathbb T^3
\end{equation}
with some positive constant $C$ satisfies the condition
\begin{equation}\label{lip1}
|\varphi ({\bf t})-\varphi({\bf 0})|\leq C|{\bf t}|^\gamma,
\quad {\bf t}=(t_1,t_2,t_3)\in \mathbb T^3,
\end{equation}
provided that
\begin{equation}\label{c-2-1}
h_i\in {\rm Lip}(p_i, \gamma), \quad 1\leq p_i \leq\infty,\quad
i=0,1,2,3,\quad {\rm and} \quad  \sum_{i=0}^3 \frac1{p_i}\leq 1.
\end{equation}
To prove (\ref{lip1}) we fix ${\bf t}=(t_1,t_2,t_3)\in\mathbb T^3$ and denote
\begin{equation}\label{c-2-2}
\overline{h}_i(u)=h_i(u-t_i)-h_i(u),\quad i=1,2,3.
\end{equation}
Since $h_i\in {\rm Lip}(p_i, \gamma)$ we have
\begin{equation}\label{c-2-3}
\|\overline{h}_i\|_{p_i}\leq C_i|{\bf t}|^\gamma,\quad i=1,2,3.
\end{equation}
By (\ref{c-2-0}) and (\ref{c-2-2})
$$
\varphi ({\bf t})=\int_\mathbb{T}h_0(u)\prod_{i=1}^3
\left(\overline{h}_i(u)+h_i(u)\right)du=
\varphi ({\bf 0}) + W.
$$
Each of the seven integrals comprising $W$ contains
at least one function $\overline{h}_i$, and in view
of (\ref{c-2-3}), can be estimated as follows
$$ \left|\int_\mathbb{T}
h_0(u)\overline{h}_1(u)h_2(u)h_3(u)du\right|\leq
\|h_0\|_{L^{p_0}}\|\overline{h}_1\|_{L^{p_1}}\|h_2\|_{L^{p_2}}
\|h_3\|_{L^{p_3}}\leq C|{\bf t}|^\gamma.
$$
This completes the proof of Theorem \ref{th02}.
\end{proof}

\begin{proof}[Proof of Theorem \ref{NT}]
For given $\alpha_i>0$, $i=1,2$, satisfying
$\alpha_1+\alpha_2<1/2$ we set
$$
\frac1{p_1}=\frac14+\frac{\alpha_1-\alpha_2}2 \quad {\rm and} \quad
\frac 1{p_2}=\frac14+\frac{\alpha_2-\alpha_1}2.
$$
It is easy to check that such defined $p_1$ and $p_2$ satisfy the
conditions
 $$
 p_1>1, \ p_2>1, \ \alpha_1 p_1<1, \ \alpha_2 p_2<1, \quad {\rm and}
 \quad \frac1{p_1}+\frac1{p_2}=\frac12.
 $$
Hence, according to Lemma \ref{lem4},
$f_i\in {\rm Lip}(p_i, \gamma)$,
$i=1,2$,
with
$$
\gamma =\frac1{p_1}-\alpha_1 =\frac1{p_2}-\alpha_2
=\frac 14-\frac{\alpha_1+\alpha_2}2.
$$
Applying Theorem \ref{th02} with $p=p_1$, $q=p_2$, $f=f_1$ and $g=f_2$
we get (\ref{t20})
with $\gamma $ as in (\ref{d2}).
\end{proof}

\section*{Acknowledgment}

The research of Mamikon S. Ginovyan was partially supported by
NSF Grant \#DMS-0706786.
The authors would like to thank the Editor and the anonymous
referees for their careful review of the manuscript and valuable comments.

\bibliographystyle{elsarticle-harv}

\end{document}